\documentclass[11pt,a4paper]{article}
\usepackage{amscd}
\usepackage[T1]{fontenc}
\usepackage[english]{babel}
\usepackage[ansinew]{inputenc}
\usepackage[psamsfonts]{amssymb}
\usepackage{amsmath}
\usepackage[all]{xy}
\usepackage{graphicx}
\usepackage{graphics}
\usepackage{float}
\usepackage{xspace}
\usepackage{tikz}
\usepackage{subfigure}

\newtheorem{definition}{Definition}[section]
\newtheorem{theorem}{Theorem}[section]

\newtheorem{corollary}{Corollary}[section]
\newtheorem{lemma}{Lemma}[section]

\newenvironment{proof}[1][Proof]{\noindent\textbf{#1.} }{\ \rule{0.5em}{0.5em}}
\begin{document}
\title{\textbf{On the diameter and incidence energy of iterated total graphs}}

\author{Eber Lenes$^{a}$, Exequiel Mallea-Zepeda$^{b}$\thanks{%
Corresponding author. E-mail addresses: elenes@unisinucartagena.edu.co (Eber Lenes),
emallea@uta.cl (Exequiel Mallea), mrobbiano@ucn.cl (Mar\'ia Robbiano), jrodriguez01@ucn.cl (Jonnathan Rodr\'iguez).}, Mar\'ia Robbiano$^{c}$, Jonnathan Rodr\'iguez$^{c}$ 
\\
$^{a}${\small Departamento de Investigaciones, Universidad del Sin\'u. El\'ias Bechara Zain\'um}, {\small Cartagena, Colombia.}\\
$^{b}${\small Departamento de Matem\'{a}tica, Universidad  de Tarapac\'a}, {\small Arica, Chile.}\\
$^{c}${\small Departamento de Matem\'{a}tica, Universidad Cat\'olica del Norte}, {\small Antofagasta, Chile.}
}
\date{}
\maketitle

\begin{abstract}
\noindent
The total graph of $G$, $\mathcal T(G)$ is the graph whose set of vertices is the union of the {sets of} vertices and edges of $G$, where two vertices are adjacent if and only if they stand for either incident or adjacent elements in $G$. 
{For $k\geq2$, the $k\text{-}th$ iterated total graph of $G$, $\mathcal{T}^k(G)$, is defined recursively as} $\mathcal{T}^k(G)=\mathcal{T}(\mathcal{T}^{k-1}(G)).$  If $G$ is a connected
graph its diameter is the maximum distance between
any pair of vertices in $G$. The incidence energy $IE(G)$ of $G$ is the sum of the
singular values of the incidence matrix of $G$. In this paper for a given integer $k$ we establish a necessary and sufficient condition under which $diam(\mathcal{T}^{r+1}(G))>k-r,$ \ $r\geq0$.
In addition, bounds for the incidence energy of the iterated graph $\mathcal{T}^{r+1}(G)$ are obtained, provided $G$ to be a regular graph. Finally, {new
families of non-isomorphic cospectral graphs are exhibited.}
\end{abstract}

\textit{\bf{2010 AMS Classification}:}\  05C12, 05C50, 05C75, 05C76, 15A18

{\bf{Keyword}}: \noindent Total graph; Line graph; Diameter; Incidence energy; Regular graph.
\thispagestyle{empty} \baselineskip=0.30in
\newpage
\clearpage

\section{Introduction}
Let $G$ be a simple connected graph on $n$ vertices. Let $v_1,v_2,{\ldots,} v_n$ be the vertices of $G$. An edge $e$ with end vertices $u$ and $v$ will be denoted by $(uv)$. {Sometimes, after a labeling of the vertices of $G$, a vertex $v_i$ is simply referred by its label $i$ and an edge $v_{i}v_{j}$ is simply referred as $(ij)$. The incident vertices to the edge $(ij)$ are $i$ and $j$.} The distance between two vertices $v_i$ and $v_j$
in $G$ is equal to the length of the shortest path in $G$ joining $v_i$ and $v_j$, 
denoted by $d_G(v_i,v_j)$. The diameter of $G$,
denoted by $diam(G)$, is the maximum distance between any pair of vertices in $G$. 
The above distance provides the simplest and most natural metric in graph
theory, and its study has recently had increasing interest in discrete mathematics research. 
As usual, $K_n$, $P_n$, $C_n$, and $S_n$ denote, respectively, the complete graph, the path, the cycle, and the star of $n$ vertices.

The line graph of $G$, denoted by $\mathcal{L}(G)$, is the graph whose vertex set are the edges in $G$, where two vertices are adjacent
if the corresponding edges in $G$ have a common vertex. The $k\text{-}th$ iterated line graph of $G$
is defined recursively as
$\mathcal{L}^k(G)=\mathcal{L}(\mathcal{L}^{k-1}(G))$, $k\geq2$, where $\mathcal{L}(G)=\mathcal{L}^1(G)$
and $G=\mathcal{L}^0(G)$. 
Metric properties of the line graph have recently been extensively
studied 
in the mathematical literature \cite{DM32,DM33,DM35,DM34,DM31}, and it found remarkable applications in chemistry \cite{DM37,DM38,DM36,DM39}.

The total graph of $G$, denoted by $\mathcal{T}(G)$, is the graph whose vertex set corresponds {union of set of vertices and edges of $G$, where, two vertices are adjacent if their
corresponding elements are either adjacent or incident in $G$. }
The $k\text{-}th$ iterated total graph of $G$ is defined recursively as
$\mathcal{T}^k(G)=\mathcal{T}(\mathcal{T}^{k-1}(G))$, $k\geq2$, where $\mathcal{T}(G)=\mathcal{T}^1(G)$ and $G=\mathcal{T}^0(G)$.

The incidence matrix, $I(G),$ is the $n\times m$ matrix
whose $(i,j)$-entry is $1$ if $v_i$ is incident to $e_j$ and $0$ otherwise. It is known that

\begin{equation}
I(G)I^T(G)=Q(G)\label{13}
\end{equation}
and
\begin{equation}
I^T(G)I(G)=2I_{m}+A(\mathcal{L}(G))\label{0013}
\end{equation}
where $Q(G)$ is the signless Laplacian matrix of $G$ and $A(\mathcal{L}(G))$ is the adjacency matrix of the line graph of $G$.

In \cite{p6,p7}, the authors have introduced the concept of the incidence energy $IE(G)$ of $G$ as the sum of the
singular values of the incidence matrix of $G$. It is well known that the singular values of a matrix
$M$ are the nonzero square roots of $MM^T$ or $M^TM$ {as}  these matrices have the same nonzero
eigenvalues. From these facts and (\ref{13})  
it follows that

$$IE(G)=\sum_{i=1}^{n}\sqrt{q_i},$$
where $q_{1},q_{2},{\ldots,} q_{n}$ are the signless Laplacian eigenvalues of $G$. 
 
In \cite{p6}, Gutman, Kiani, Mirzakhah, and Zhou, proved that

\begin{theorem}
\cite{p6}
Let $G$ be a regular graph on $n$ vertices 
of degree $r$. Then,
\begin{equation*}
IE(\mathcal{L}^{k+1}(G))\leq\dfrac{n_k(r_k-2)}{2}\sqrt{2r_k-4}+\sqrt{4r_k-4}+\sqrt{(n_k-1)[(3r_k-4)(n_k-1)-r_k]},
\end{equation*}
where $n_k$ and $r_k$ are the order and degree of $\mathcal{L}^{k+1}(G)$, respectively. Equality holds if and only if $\mathcal{L}^{k}(G)\cong K_n$.
\end{theorem}

This paper is organized as follows. In Section 2, we establish conditions on a graph $G$, in order to the diameter of $\mathcal{T}(G)$ {does not exceed $k$, \ $k\geq2$.} Also 
conditions {in order to the diameter of $\mathcal{T}(G)$ to be greater than or equal to $k$, \ $k\geq3$,}  are established. Moreover, { we establish a necessary {condition} so that the diameter of $\mathcal{T}^{r+1}(G)$ does not be greater than or equal to $k-r$, \ $k\geq2$, $r\geq 0$.} In Section 3, {we derive upper and lower bounds} on the incidence energy for {the} iterated total graphs of regular graphs.
Aditionally, we construct {some} new families of nonisomorphic cospectral graphs.

\section{Diameter of total graphs}

In this section we establish structural conditions for a graph $G$, so that the diameter of $\mathcal{T}(G)$ does not exceed $k$ and it be no less than $k$, {\ $k\geq3$.}

\subsection{Main results}

{Before proceeding we need the following definitions.}

\begin{definition}
A path $P$ with end vertices $u$ and $v$ {in} a connected graph $G$ is called a diameter path {of} $G$ if {$d_G(u,v)=diam(G)$ and $P$ is a path with $diam(G)+1$ vertices.}
\end{definition}

\begin{definition}
{A subgraph $H$ of a connected graph $G$ is called a diameter subgraph of $G$ if $H$ has a diameter path of $G$ as a subgraph.}
\end{definition}

\begin{definition}
{The Lollipop $Lol_{n,g}$ is the graph obtained from a cycle with $g$ vertices by identifying one of its vertices with a vertex of a path of length $n-g.$ Note that this graph has $n$ vertices and diameter  $n-g+1+\left\lfloor \frac { g }{ 2 }  \right\rfloor .$}
{In \cite{cardoso} was conjectured that $Lol_{3,n-3}$ is the nonbipartite graph on $n$ vertices with minimum smallest signless Laplacian eigenvalue $q_n.$}
\end{definition}

\begin{figure}[h]
\begin{eqnarray*}
\begin{tikzpicture};
    \tikzstyle{every node}=[draw,circle,fill=black,minimum size=3.5pt,inner sep=0.2pt]
    \draw
                  (-1,0) node (1) [label=below:] {}
                  (0,-1) node (2) [label=below:] {}
                  (1,0) node (3) [label=below:] {}
                  (0,1) node (4) [label=below:] {}
                  (2,0) node (5) [label=below:] {}
                  (3,0) node (6) [label=below:] {}
                  (4,0) node (7) [label=above:] {}
                  (5,0) node (8) [label=below:] {};
        \draw (1)--(2);\draw (2)--(3);\draw (3)--(4);\draw (4)--(1);\draw (3)--(5);\draw (5)--(6);\draw (6)--(7);\draw (7)--(8);
            \end{tikzpicture}
\end{eqnarray*}
\caption{A Lollipop $Lol_{8,4}$.}%
\end{figure}
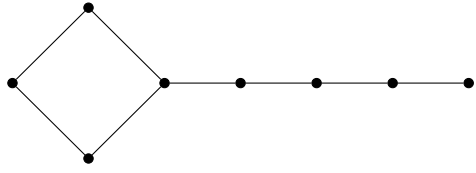

\begin{lemma}\label{lem2}
Let $G$ be a connected graph on $n$ vertices. Then either $diam(\mathcal{T}(G))=diam(G)$ or $diam(\mathcal{T}(G))=diam(\mathcal{L}(G))$ or $Lol_{l+diam(G)+1,2l+1}$ is a diameter subgraph of $G$ for some $1\leq l\leq diam(G)${,} where $diam(\mathcal{T}(G))=diam(G)+1$.
\end{lemma}

\begin{proof}
  The result can be easily verified for graphs of order $n\leq3$. In consequence, we assume, $n>3$. If $diam(G)< diam(\mathcal{T}(G))$, $diam(\mathcal{L}(G))< diam(\mathcal{T}(G))$ and $Lol_{l+diam(G)+1,2l+1}$ is not a diameter subgraph of $G$ for all $1\leq l\leq diam(G)$ where $diam(\mathcal{T}(G))=diam(G)+1$. We claim that $d_{\mathcal{T}(G)}(u,v)<diam(\mathcal{T}(G))$ for any pair of vertices $u,v$ in $\mathcal{T}(G)$. Let $v_{a}$, $v_{b}$ be two vertices in $\mathcal{T}(G)$ such that $d_{\mathcal{T}(G)}(v_{a},v_{b})\geq diam(\mathcal{T}(G))$.
  Then there must exists $diam(\mathcal{T}(G))-1$ vertices in $\mathcal{T}(G)$, say $v_1, v_2,{\ldots,} v_{diam(\mathcal{T}(G))-1}$, such that $v_i$ is adjacent to $v_{i+1}$,
  $i=1,2,{\ldots,} diam(\mathcal{T}(G))-2$, where $v_{a}$ is adjacent to $v_1$. 
  \begin{enumerate}
  \item If $v_{a}, v_{b}$ are vertices of $G$ then we can assume that $v_i$ is a vertex of $G$,  $i=1,2,{\ldots,}diam(\mathcal{T}(G))-1$. Then, $diam(G)\geq diam(\mathcal{T}(G))$, which is impossible.
  \item If $v_{a}, v_{b}$ are vertices of the line graph of $G$, that is to say if $v_{a}, v_{b}$ are edges of $G$, then we can assume that $v_i$ is an edge of $G$, $i=1,2,{\ldots,}diam(\mathcal{T}(G))-1$. Thus, $diam(\mathcal{L}(G))\geq diam(\mathcal{T}(G))$, which is impossible.
  \item If $v_{a}$ is a vertex of $G$ and $v_{b}$ is an edge of $G$, say $v_{b}=(b_1b_2)$, we can assume without loss of generality $v_i$ is a vertex of $G,$  $i=1,2,{\ldots,}diam(\mathcal{T}(G))-1,$ or $v_i$ is an edge of $G,$  $i=1,2,{\ldots,}diam(\mathcal{T}(G))-1$. Suppose $v_i$ is a vertex of $G,$  $i=1,2,{\ldots,}diam(\mathcal{T}(G))-1$. Then, $d_G(v_a,b_1)\geq diam(\mathcal{T}(G))-1$. Since $diam(G)< diam(\mathcal{T}(G))$ then $d_G(v_{a},b_1)=diam(\mathcal{T}(G))-1=diam(G)$. In addition, $d_G(v_a,b_2)=diam(G)$. Thereby, $Lol_{l+diam(G)+1,2l+1}$ is a diameter subgraph of $G$ \ some $1\leq l\leq diam(G)$ where $diam(\mathcal{T}(G))=diam(G)+1$, which is impossible. Otherwise, suppose $v_i$ is an edge of $G,$  $i=1,2,{\ldots,}diam(\mathcal{T}(G))-1$, with $v_{diam(\mathcal{T}(G))-1}=(cb_1)$ then $d_G(v_{a},b_1)\geq diam(\mathcal{T}(G))-1$. Since $diam(G)<diam(\mathcal{T}(G))$ then $d_G(v_{a},b_1)=diam(\mathcal{T}(G))-1=diam(G)$. Moreover, $d_G(v_a,b_2)=diam(G)$. Thus, $Lol_{l+diam(G)+1,2l+1}$ is a diameter subgraph of $G$ for some $1\leq l\leq diam(G)$ where $diam(\mathcal{T}(G))=diam(G)+1$, which is impossible. 
  
  This is, $d_{\mathcal{T}(G)}(u,v)<diam(\mathcal{T}(G))$ for any pair of vertices $u,v$ in $\mathcal{T}(G)$, which is a contradiction.

  \end{enumerate}
\end{proof}

Before proceeding we need establish the following facts {about the} line graphs.

\begin{lemma}\label{lem3} \cite{DM30}
If $H$ is an induced subgraph of $G$ then $\mathcal{L}(H)$ is an induced subgraph of $\mathcal{L}(G)$.
\end{lemma}

{Let denote by $F_1^k$ the path of length $k+1$, $k\geq2$. Let $v_1,v_2,{\ldots,}v_{k+2}$ be the vertices of $F_1^k$ so that for $i=1,2,{\ldots,}k+1,$ $v_i$ is adjacent to $v_{i+1}$.}

{Let $F_2^k$ and $F_3^k$ be the graphs obtained from $F_1^{k+1}$ by adding the edge $(v_1v_3)$, and the edges $(v_1v_3)$ and $(v_{k+1}v_{k+3}),$ respectively. Note that $F_1^{k+1}$, $F_2^{k}$ and $F_3^{k}$ have diameter $k+1.$ (see Fig. 2).}

\begin{figure}[t]
\begin{eqnarray*}
\begin{tikzpicture}\node at (-0.66666,2) {\textbf{{$F_{1}^{k}$}}};\node at (-4.66666,-1) {\textbf{{$F_{2}^{k}$}}};\node at (3.25,-1) {\textbf{{$F_{3}^{k}$}}};\node at (1.2,2.7) {$v_{k+1}$};\node at (2.2,2.7) {$v_{k+2}$};\node at (-3.8,-0.3) {$v_{k+1}$};\node at (-2.8,-0.3) {$v_{k+2}$};\node at (-1.8,-0.3) {$v_{k+3}$};\node at (4.2,-0.3) {$v_{k-1}$};\node at (6.1,-0.3) {$v_{k+1}$};\node at (5.2,-0.3) {$v_{k}$};\node at (6.7,0.73) {$v_{k+2}$};\node at (6.7,-0.73) {$v_{k+3}$};
    \tikzstyle{every node}=[draw,circle,fill=black,minimum size=3.5pt,inner sep=0.2pt]
    \draw
                  (-3.5,3) node (1) [label=below:$v_{1}$] {}
                  (-2.5,3) node (2) [label=below:$v_{2}$] {}
                  (-1.5,3) node (3) [label=below:$v_{3}$] {}
                  (-0.33333,3) node (0) [label=below:] {}
                  (-0.66666,3) node (a) [label=below:] {}
                  (-1,3) node (b) [label=below:] {}
                  (0.2,3) node (4) [label=below:$v_{k}$] {}
                  (1.2,3) node (5) [label=below:] {}
                  (2.2,3) node (6) [label=below:] {}
                  (-8,0.5) node (7) [label=above:$v_{1}$] {}
                  (-8,-0.5) node (8) [label=below:$v_{2}$] {}
                  (-7.5,0) node (9) [label=below:$v_{3}$] {}
                  (-6.5,0) node (10) [label=below:$v_{4}$] {}
                  (-5.5,0) node (11) [label=below:$v_{5}$] {}
                  (-4.33333,0) node (d) [label=below:] {}
                  (-4.66666,0) node (f) [label=below:] {}
                  (-5,0) node (g) [label=below:] {}
                  (-3.8,0) node (12) [label=below:] {}
                  (-2.8,0) node (13) [label=below:] {}
                  (-1.8,0) node (14) [label=below:] {}
                  (-0.3,0.5) node (15) [label=above:$v_{1}$] {}
                  (-0.3,-0.5) node (16) [label=below:$v_{2}$] {}
                  (0.3,0) node (17) [label=below:$v_{3}$] {}
                  (1.3,0) node (18) [label=below:$v_{4}$] {}
                  (2.3,0) node (19) [label=below:$v_{5}$] {}
                  (2.875,0) node (r) [label=below:] {}
                  (3.25,0) node (t) [label=below:] {}
                  (3.675,0) node (b) [label=below:] {}
                  (4.2,0) node (20) [label=below:] {}
                  (5.2,0) node (21) [label=below:] {}
                  (6.2,0) node (22) [label=below:] {}
                  (6.7,0.5) node (23) [label=above:] {}
                  (6.7,-0.5) node (24) [label=below:] {};
        \draw (1)--(2);\draw (2)--(3);\draw (0) (a) (b);\draw (4)--(5);\draw (5)--(6);\draw (7)--(8);\draw (8)--(9);\draw (9)--(7);\draw (9)--(10);\draw (10)--(11);\draw (12)--(13);\draw (13)--(14);\draw (15)--(16);\draw (17)--(16);\draw (15)--(17);\draw (17)--(18);\draw (18)--(19);\draw (20)--(21);\draw (21)--(22);\draw (22)--(23);\draw (23)--(24);\draw (24)--(22);
            \end{tikzpicture}
\end{eqnarray*}
\caption{Graphs $F_1^{k}$, $F_2^{k}$ and $F_3^{k}$.}%
\end{figure}
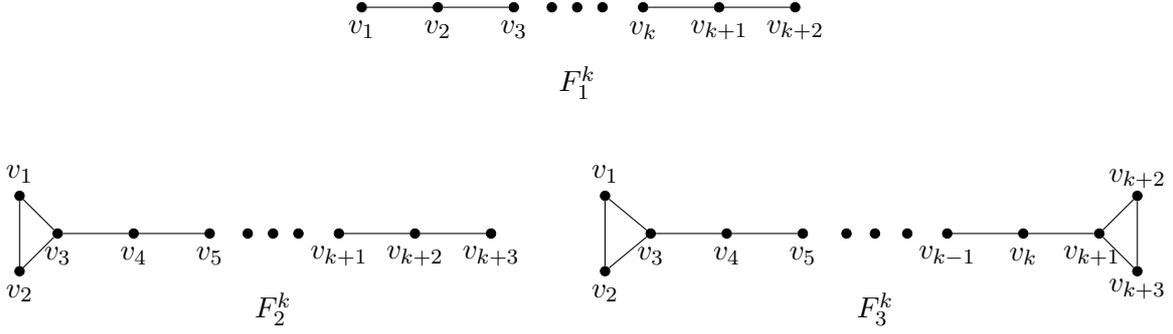

\begin{theorem}\label{teo3}
 Let $k\geq2$. Then $diam(\mathcal{T}(G))\leq k$ if and only  {if all the following conditions fail to hold}
 \begin{itemize}
 \item $F_1^{k+1}$, ${F_2}^k$ and ${F_3}^k$ are induced subgraphs of $G$, and
 \item $F_1^k$ is a diameter path of $G$, and
 \item $Lol_{l+k+1,2l+1}$ is a diameter subgraph of $G$ for some $1\leq l\leq k$ where $diam(G)=k$.
 \end{itemize}
\end{theorem}

\begin{proof}
The result can be easily verified for graphs of order $n\leq3$. In consequence, we assume, $n>3$. 
Let $k\geq2$ and suppose that $diam(\mathcal{T}(G))\leq k$. If $F_1^k$ is a diameter path of $G$, by
Lemma \ref{lem2}, $diam(\mathcal{T}(G))\geq diam(F_1^k)=k+1>k$, which is a contradiction. Suppose that $F_2^k$ is an induced subgraph
of $G$, by Lemma \ref{lem3}, $\mathcal{L}(F_2^k)$ is an induced subgraph of $\mathcal{L}(G)$. Thus, $\mathcal{L}(F_2^k)$ is an induced subgraph of $\mathcal{T}(G)$. Therefore, $diam(\mathcal{T}(G))\geq diam(\mathcal{L}(F_2^k))=k+1$, which is a contradiction. Similarly, we  {prove} that $F_1^{k+1}$ and $F_3^k$ are not induced subgraphs of $G$. Moreover, if $Lol_{l+k+1,2l+1}$ is a diameter subgraph of $G$ for some $1\leq l\leq k$ where $diam(G)=k$ then $diam(\mathcal{T}(G))\geq diam(\mathcal{T}(Lol_{l+k+1,2l+1}))=k+1$, which is a contradiction.
Conversely, suppose that $k\geq2$ and $diam(\mathcal{T}(G))>k$. By Lemma \ref{lem2}, we have the following cases

\begin{itemize}
\item {\bf{Case 1:}} Suppose $diam(\mathcal{T}(G))=diam(G)$.
If $diam(G)=k+1$ then $F_1^k$ is a diameter path in $G$, a contradiction.
If $diam(G)\geq k+2$ then $F_1^{k+1}$ is an induced subgraph of $G$, a contradiction.

\item {\bf{Case 2:}} Suppose $diam(\mathcal{T}(G))=diam(\mathcal{L}(G))$ then there must exists $k+2$ edges in $G$,
say $e_1, e_2,{\ldots,} e_{k+2}$, such that $e_i$ is incident to $e_{i+1}$ in $G$, $i=1,2,{\ldots,}k+1$, and $d_{\mathcal{L}(G)}(e_1,e_{k+2})=k+1$.
Let $e_i=(v_iv_{i+1})$, $i=1,2,{\ldots,} k+1$. Then

$(i)$ If $v_1$ is not adjacent to $v_3$ and $v_{k+1}$
is not adjacent to $v_{k+3}$ then $F_1^{k+1}$ is an induced subgraph of $G$, a contradiction.

$(ii)$ If $v_1$ is adjacent to $v_3$ (or $v_{k+1}$
is adjacent to $v_{k+3}$) then $F_2^{k}$ is an induced subgraph of $G$, a contradiction.

$(iii)$ If $v_1$ is adjacent to $v_3$ and $v_{k+1}$
is adjacent to $v_{k+3}$ then $F_3^{k}$ is an induced subgraph of $G$,  a contradiction.

\item{\bf{Case 3:}} Suppose that $Lol_{l+diam(G)+1,2l+1}$ is a diameter subgraph of $G$ for some $1\leq l\leq diam(G)$ where $diam(\mathcal{T}(G))=diam(G)+1$. If $diam(\mathcal{T}(G))=k+1$ then $Lol_{l+k+1,2l+1}$ is a diameter subgraph of $G$ for some $1\leq l\leq k$ where $diam(G)=k$, a contradiction. If $diam(\mathcal{T}(G))=k+2$ then $F_1^k$ is a diameter path of $G$, a contradiction.
If $diam(\mathcal{T}(G))\geq k+3$ then $F_1^{k+1}$ is an induced subgraph of $G$, a contradiction.

\end{itemize}
\end{proof}

An equivalent result to Theorem \ref{teo3} is given below.

\begin{theorem} \label{teo2}
 Let $k\geq2$. For a connected graph $G$, $diam(\mathcal{T}(G))> k$ if and only if some of the following conditions is verified
\begin{itemize}
\item $F_1^{k+1}$ or ${F_2}^k$ or ${F_3}^k$ is an induced subgraph of $G$, or
\item $F_1^k$ is a diameter path of $G$, or
\item $Lol_{l+k+1,2l+1}$ is a diameter subgraph of $G$ for some $1\leq l\leq k$ where $diam(G)=k$.
\end{itemize}
\end{theorem}

\begin{theorem}
For a connected graph $G$, $diam(\mathcal{T}(G))=1$ if and only if $G\cong K_2$.
\end{theorem}

\begin{proof}
  If $G\cong K_2$ then all the vertices of $\mathcal{T}(G)$ are adjacent, this is, $\mathcal{T}(G)=K_3$. Thus, $diam(\mathcal{T}(G))=diam(K_3)=1$.

  Conversely, let $diam(\mathcal{T}(G))=1$  and $G\ncong K_2$ then there exists two different edges in $G$, say $e_i=(xy)$
  and $e_j=(yz)$. Thus, $d_{\mathcal{T}(G)}(e_j,x)=2$. Therefore, $diam(\mathcal{T}(G))>1$, a contradiction. 
\end{proof}

\noindent{ The following results characterize some subgraphs of $\mathcal{T}(G)$ according to the diameter of either $\mathcal{L}(G)$ or $G$. Evidently, the diameter of $\mathcal{T}(G)$ is zero if and only if $G\cong K_1$. }

\begin{lemma}\label{lem5}
For a connected graph $G$, if $diam(\mathcal{L}(G))=diam(G)=2$ then some of the following conditions is verified
\begin{itemize}
\item $K_4-e$ or $Lol_{4,3}$, or $C_4$ is an induced subgraph of $G$, or
\item $C_5$ is a diameter subgraph of $G$.
\end{itemize}
\end{lemma}

\begin{proof}
Let $G$ be a connected graph such that $diam(\mathcal{L}(G))=2$ then there exists $4$ vertices, say $u_1, u_2, u_3, u_4,$  such that $u_i$ is adjacent to $u_{i+1}$, for $i=1,2,3$.

$(i)$ If $u_1$ is not adjacent to $u_3$, $u_2$ is not adjacent to $u_4$ and $u_1$ is not adjacent to $u_4$ then $P_4$ is an induced subgraph of $G$. Moreover, since $diam(G)=2$ then $P_3:u_1vu_4$ is a diameter path of $G$ for some $v$. Hence, $C_5$ is a diameter subgraph of $G$.

$(ii)$ If $u_1$ is adjacent to $u_3$ (or $u_2$ is adjacent to $u_4$) and $u_1$ is not adjacent to $u_4$ then $Lol_{4,3}$ is an induced subgraph of $G$.

$(iii)$ If $u_1$ is adjacent to $u_3$, $u_2$ is adjacent to $u_4$ and $u_1$ is not adjacent to $u_4$ then $K_4-e$ is an induced subgraph of $G$.

$(iv)$ If $u_1$ is adjacent to $u_4$, $u_1$ is not adjacent to $u_3$ and $u_2$ is not adjacent to $u_4$ then $C_4$ is an induced subgraph of $G$. 
\end{proof}

\begin{theorem}
Let $G$ be a connected graph that such {all the following conditions fail to hold}
\begin{itemize}
\item $K_4-e$, $Lol_{4,3}$ and $C_4$ is an induced subgraph of $G$, and
\item $C_5$ is a diameter subgraph of $G$.
\end{itemize}
Then $diam(\mathcal{T}(G))=2$ if and only if $G\cong K_n$ or $G\cong S_n$.
\end{theorem}

\begin{proof}
Suppose that none of the three graphs
$K_4-e$, $Lol_{4,3}$ and $C_4$ are induced subgraphs of $G$ 
and that $C_5$ is not a diameter subgraph of $G$. By Lemma \ref{lem5} does not occur that $diam(\mathcal{L}(G))=diam(G)=2$. Now, if $diam(\mathcal{T}(G))=2$ by Lemma \ref{lem2} some of the following cases is verified
\begin{itemize}
  \item {\bf{Case 1:}} $diam(G)=2$. Then $diam(\mathcal{L}(G))=1$ thus any couple of edges of $G$ are incidents. In consequence, $G\cong K_3$ or $G\cong S_n$.
  Since $diam(G)=2$, we concluded that $G\cong S_n$.
  \item {\bf{Case 2:}} $diam(\mathcal{L}(G))=2$. Then $diam(G)=1$ thus any couple of vertices of $G$ are adjacents.
  Therefore, $G\cong K_n$.
  \item {\bf{Case 3:}} $Lol_{3,3}$ is a diameter subgraph of $G$. Then, $diam(G)=1$. Thus, $G\cong K_n$.
\end{itemize}
Conversely, let $G\cong K_n$. Then any couple of vertices in $G$ are at distance 1. Let $e_i=(xy)$ and $e_j=(zw)$ be two different edges in $G$. If $e_i$ and $e_j$ are incident edges, then $d_{\mathcal{L}(G)}(e_i,e_j)=1$.
Otherwise, since $e_k=(yz)$ is an edge in $G$ we have $d_{\mathcal{L}(G)}(e_i,e_j)=2$. Finally, let $e=(xy)$ be an edge of $G$ and let $v$ be a vertex of $G$, if $v=x$ or $v=y$ then $d_{\mathcal{T}(G)}(e,v)=1$. Otherwise, since $xv$ is an edge of $G$ then $d_{\mathcal{T}(G)}(e,v)=2$. Hence,
$diam(\mathcal{T}(G))=2$.

If $G\cong S_n$, then all the edges of $G$ are incidents to a common vertex. Therefore, all vertices are {pairwise incident}
in $\mathcal{L}(G)$ and thus $\mathcal{L}(G)\cong K_{n-1}$. Hence, $diam(\mathcal{L}(G))=1$. Moreover, all $n-1$ vertices are adjacents to a common vertex in
$G$. Then, $diam(G)=2$. Finally, since $diam(\mathcal{L}(G))=1$ then any edge of $G$ and any vertex of $G$ are to distance less than or equal to 2. Therefore, $diam(\mathcal{T}(G))=2$. 
\end{proof}


\subsection{Results for iterated total graphs}

\begin{theorem}\label{teo7}
Let $r\geq1$ and $k\geq 4r+3$. Let $G$ be a connected graph such that 
\begin{equation*}
diam(\mathcal{T}^{r+1}(G))> k-r.
\end{equation*}
Then $F_1^{k-4r-1}$ is an induced subgraph of $G$.
\end{theorem}

\begin{proof}
Suppose $diam(\mathcal{T}^{r+1}(G))> k-r$. By Theorem \ref{teo2},
\begin{itemize}
\item $F_1^{k-r+1}$ or ${F_2}^{k-r}$ or ${F_3}^{k-r}$ is an induced subgraph of $\mathcal{T}^{r}(G)$, or
\item $F_1^{k-r}$ is a diameter path of $\mathcal{T}^{r}(G)$, or
\item $Lol_{l+k-r+1,2l+1}$ is a diameter subgraph of $\mathcal{T}^r(G)$ for some $1\leq l\leq k-r$ where $diam(\mathcal{T}^{r}(G))=k-r$.
\end{itemize}
Moreover,
\begin{itemize}
\item[a)] If $F_1^{k-r+1}$ is an induced subgraph of $\mathcal{T}^r(G)$ then either $F_1^{k-r-2}$ or $F_1^{k-r+1}$ is an induced subgraph of $\mathcal{T}^{r-1}(G)$ or $F_1^{k-r+1}$ is an induced subgraph of $\mathcal{L}(\mathcal{T}^{r-1}(G))$. Since $\mathcal{L}(F_1^{k-r+2})=F_1^{k-r+1}$ then $F_1^{k-r-2}$ is an induced subgraph of $\mathcal{T}^{r-1}(G)$.

\item[b)] If $F_2^{k-r}$ is an induced subgraph of $\mathcal{T}^r(G)$ then $F_1^{k-r}$ is an induced subgraph of $\mathcal{T}^{r}(G)$. By $a)$,
$F_1^{k-r-3}$ is an induced subgraph of $\mathcal{T}^{r-1}(G)$.


\item[c)] If $F_3^{k-r}$ is an induced subgraph of $\mathcal{T}^{r}(G)$ then $F_1^{k-r-1}$ is an induced subgraph of $\mathcal{T}^{r}(G)$. By $a)$,
$F_1^{k-r-4}$ is an induced subgraph of $\mathcal{T}^{r-1}(G)$.


\item[d)] If $F_1^{k-r}$ is a diameter path of $\mathcal{T}^r(G)$ then either $F_1^{k-r}$ or $F_1^{k-r-1}$ is a diameter path of $\mathcal{T}^{r-1}(G)$ or $F_1^{k-r}$ is a diameter path of $\mathcal{L}(\mathcal{T}^{r-1}(G))$. Since $\mathcal{L}(F_1^{k-r+1})=F_1^{k-r}$ then $F_1^{k-r-1}$ is an induced subgraph of $\mathcal{T}^{r-1}(G)$.


\item[e)] If $Lol_{l+k-r+1,2l+1}$ is a diameter subgraph of $\mathcal{T}^r(G)$ for some $1\leq l\leq k-r$ where $diam(\mathcal{T}^r(G))=k-r$, then $F_1^{k-r-1}$ is a diameter path of $\mathcal{T}^r(G)$. By $d)$, $F_1^{k-r-2}$ is a diameter path of $\mathcal{T}^{r-1}(G)$.

\end{itemize}

Therefore, $F_1^{k-r-4}$ is an induced subgraph of $\mathcal{T}^{r-1}(G)$. By a), $F_1^{k-r-7}$ is an induced subgraph of $\mathcal{T}^{r-2}(G)$. Then, $F_1^{k-r-10}$  is an induced subgraph of $\mathcal{T}^{r-3}(G)$. Following this process we concluded $F_1^{k-4r-1}$ is an induced subgraph of $G$.
\end{proof}

\begin{theorem}
Let $r\geq1$ and $k\geq2r+2$. Let $F_1^{k-2r}$ be an induced subgraph of $G$ then 
\begin{equation*}
diam(\mathcal{T}^{r+1}(G))>k-r.
\end{equation*}
\end{theorem}

\begin{proof}
Suppose $F_1^{k-2r}$ is an induced subgraph of $G$ then $F_1^{k-2r+1}$ is an induced subgraph of $\mathcal{T}(G)$. Moreover, $F_1^{k-2r+2}$
is an induced subgraph of $\mathcal{T}^2(G)$. Following this process we concluded that $F_1^{k-r-1}$ is an induced subgraph of $\mathcal{T}^{r-1}(G)$.
Then, the graph with vertices $$v_1, (v_1v_2), v_2,{\ldots,} v_{k-r},(v_{k-r}v_{k-r+1}),\\v_{k-r+1},$$ is an induced subgraph of $\mathcal{T}^r(G)$ isomorphic to $F_3^{k-r}$. By Theorem \ref{teo2}, $diam(\mathcal{T}^{r+1}(G))>k-r$. 
\end{proof}

\subsection{Results for iterated line graphs}

Let $P_{k-1}$ be the path with vertices $v_1, v_2,{\ldots,} v_{k-1}$, where $v_i$ is adjacent to
$v_{i+1}$, $i=1,2,{\ldots,}k-2$, $k\geq3$. Let $F_4^k$ be the graph obtained from $P_{k-1}$ by joining
{two new} vertices to {the vertex} $v_1$ and another {two new} vertices {the vertex} $v_{k-1}$. {Thus} $F_4^k$ has $k+3$ vertices
and $k+2$ edges. Let $P_{k+1}$ be {a} path {on the} vertices $v_1,v_2,{\ldots,}v_{k+1}$, where $v_i$ is adjacent
to $v_{i+1}$, $i=1,2,{\ldots,}k$,\ $k\geq1$. Let $F_5^k$ be the graph obtained from $P_{k+1}$ by joining {two new} vertices to the vertex $v_{k+1}$. {Note that $F_4^{k}$ and $F_5^{k}$ have diameter $k+1.$ (see Fig. 3).}

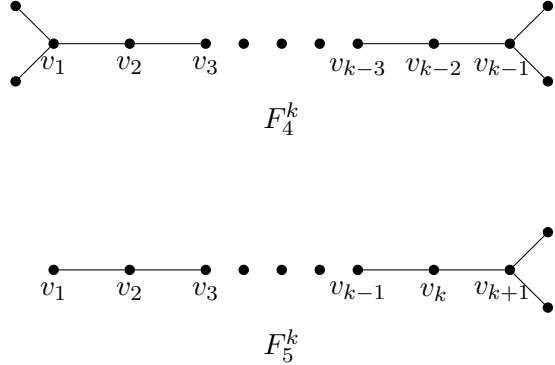
\begin{figure}[h]
\begin{eqnarray*}
\begin{tikzpicture}\node at (0,1) {\textbf{{$F_{4}^{k}$}}};\node at (0,-2) {\textbf{{$F_{5}^{k}$}}};\node at (1,1.68) {$v_{k-3}$};\node at (2,1.68) {$v_{k-2}$};\node at (2.9,1.68) {$v_{k-1}$};\node at (1,-1.28) {$v_{k-1}$};\node at (2.9,-1.28) {$v_{k+1}$};
    \tikzstyle{every node}=[draw,circle,fill=black,minimum size=3.5pt,inner sep=0pt]
    \draw
                  (-3.5,2.5) node (1) [label=below:] {}
                  (-3.5,1.5) node (2) [label=below:] {}
                  (-3,2) node (3) [label=below:$v_{1}$] {}
                  (-0.5,2) node (q) [label=below:] {}
                  (0,2) node (e) [label=below:] {}
                  (0.5,2) node (z) [label=below:] {}
                  (-2,2) node (4) [label=below:$v_{2}$] {}
                  (-1,2) node (5) [label=below:$v_{3}$] {}
                  (1,2) node (6) [label=below:] {}
                  (2,2) node (7) [label=below:] {}
                  (3,2) node (8) [label=below:] {}
                  (3.5,1.5) node (9) [label=below:] {}
                  (3.5,2.5) node (10) [label=below:] {}
                  (-3,-1) node (11) [label=below:$v_{1}$] {}
                  (-2,-1) node (12) [label=below:$v_{2}$] {}
                  (-1,-1) node (13) [label=below:$v_{3}$] {}
                  (-0.5,-1) node (p) [label=below:] {}
                  (0,-1) node (l) [label=below:] {}
                  (0.5,-1) node (z) [label=below:] {}
                  (1,-1) node (14) [label=below:] {}
                  (2,-1) node (15) [label=below:$v_{k}$] {}
                  (3,-1) node (16) [label=below:] {}
                  (3.5,-0.5) node (17) [label=below:] {}
                  (3.5,-1.5) node (18) [label=below:] {};
        \draw (1)--(3);\draw (2)--(3);\draw (3)--(4);\draw (4)--(5);\draw (6)--(7);\draw (7)--(8);\draw (8)--(9);\draw (8)--(10);\draw (11)--(12);\draw (12)--(13);\draw (14)--(15);\draw (15)--(16);\draw (16)--(17);\draw (16)--(18);
            \end{tikzpicture}
\end{eqnarray*}
\caption{Graphs $F_4^{k}$ and $F_5^{k}$.}%
\end{figure}



\begin{lemma}\cite{DM31}\label{1433}
Let $G$ be a connected graph with $n\geq3$ vertices. Let $k\geq2$. Then $diam(\mathcal{L}(G))> k$, if and only if
either $F_1^{k+1}$ or $F_2^{k}$ or $F_3^{k}$ is an induced subgraph of $G$.
\end{lemma}


Considering the Lemma 2.4, the follows results are obtained.

\begin{theorem}\label{teo2.7}
Let $r\geq1$ and $k\geq2r+3$. Let $G$ be a connected graph such that 
\begin{equation*}
diam(\mathcal{L}^{r+1}(G))>k-r.
\end{equation*}
Then $F_1^{k-2r-1}$ is an induced subgraph of $G$.
\end{theorem}

\begin{proof}
Suppose $diam(\mathcal{L}^{r+1}(G))>k-r$. By Lemma \ref{1433}, $F_1^{k-r+1}$ or $F_2^{k-r}$ or $F_3^{k-r}$ is an induced subgraph of $\mathcal{L}^{r}(G)$. Then, $F_1^{k-r-1}$ is an induced subgraph of $\mathcal{L}^r(G)$. Since $\mathcal{L}(F_1^{k-r})=F_1^{k-r-1}$.
Then, $F_1^{k-r-2}$ is an induced subgraph of $\mathcal{L}^{r-1}(G)$. Thus, $F_1^{k-r-3}$ is an induced subgraph of $\mathcal{L}^{r-2}(G)$. Then, $F_1^{k-r-4}$  is an induced subgraph of $\mathcal{L}^{r-3}(G)$. Following this process we concluded that $F_1^{k-2r-1}$ is an induced subgraph of $G$. 
\end{proof}

\begin{theorem}\label{teo2.8}
Let $1\leq r<k-1$. Let $F_1^{k+1}$ or $F_4^k$ or $F_5^k$ be an induced subgraph of $G$ then 
\begin{equation*}
diam(\mathcal{L}^{r+1}(G))>k-r.
\end{equation*}
\end{theorem}

\begin{proof}
Suppose $F_1^{k+1}$ or $F_4^k$ or $F_5^{k}$ is an induced subgraph of $G$. By Lemma \ref{lem3}, $\mathcal{L}^r(F_1^{k+1})$ or $\mathcal{L}^r(F_4^{k})$ or
$\mathcal{L}^r(F_5^{k})$ is an induced subgraph of $\mathcal{L}^r(G)$.
Moreover, $\mathcal{L}^r(F_1^{k+1})=F_1^{k-r+1}$, $\mathcal{L}^r(F_4^{k})=F_3^{k-r}$ and  $\mathcal{L}^r(F_5^{k})=F_2^{k-r}$. By Lemma \ref{1433}, $diam(\mathcal{L}^{r+1}(G))>k-r$. 
\end{proof}

\section{Energy of iterated graphs}

In this section, we derive bounds on the incidence energy of iterated total graphs of regular graphs.
Futhermore, we construct new families of nonisomorphic cospectral graphs.

\subsection{Incidence energy of iterated graphs}

The basic properties of iterated line graph sequences are summarized in the articles \cite{Buckley,Harary}.

The line graph of a regular graph is a regular graph. In particular, the line graph of a regular graph
of order $n_0$ and degree $r_0$ is a regular graph of order $n_1=\dfrac{1}{2}r_0n_0$ and degree
$r_1=2r_0-2$. Consequently, the order and degree of $\mathcal{L}^{k}(G)$ are (see \cite{Buckley,Harary}):

$n_k=\dfrac{1}{2}r_{k-1}n_{k-1}$ and $r_{k}=2r_{k-1}-2$
where $n_{k-1}$ and $r_{k-1}$ stand for the order and degree of $\mathcal{L}^{k-1}(G)$. Therefore,

\begin{equation*}
r_k=2^kr_0-2^{k+1}+2
\end{equation*}
and
\begin{equation*}
n_k=\dfrac{n_0}{2^k}\prod_{i=0}^{k-1}r_i=\dfrac{n_0}{2^k}\prod_{i=0}^{k-1}\Big(2^ir_0-2^{i+1}+2\Big).
\end{equation*}

\begin{theorem}\label{teo1}
Let $G$ be a regular graph of order $n_0$ and degree $r_0$, then for $k\geq1$
the $k\text{-}th$ iterated total graph of $G$  is a regular graph of degree $r_k$ and order $n_k$, where
\begin{enumerate}
\item $r_k=2r_{k-1}$, and 
\item $n_k=n_{k-1}\Big(\dfrac{r_{k-1}+2}{2}\Big)$.
\end{enumerate}
\end{theorem}
\begin{proof}
Let $k\geq1$. Suppose that the $(k-1)\text{-}th$ iterated total graph of $G$ is a regular graph of order $n_{k-1}$ and degree $r_{k-1}$.
\begin{enumerate}
\item
Let $v$ be a vertex of the $k\text{-}th$ iterated total graph of $G$ then
\begin{itemize}
\item {\bf{Case a)}}
If $v$ is a vertex of the $(k-1)\text{-}th$ iterated total graph of $G$ then $v$ is adjacent to $r_{k-1}$ vertices and
incident to $r_{k-1}$ edges in $\mathcal{T}^{k-1}(G)$. Thus, the degree of $v$ in $\mathcal{T}^{k}(G)$
is
\begin{equation*}
r_{k-1}+r_{k-1}=2r_{k-1}.
\end{equation*}

\item {\bf{Case b)}}
If $v$ is an edge of the $(k-1)\text{-}th$ iterated total graph of $G$ then $v$ is adjacent in each extreme to $r_{k-1}-1$ edges and incident to its two extreme vertices in $\mathcal{T}^{k-1}(G)$. Thus, the degree of $v$ in $\mathcal{T}^{k}(G)$ is
\begin{equation*}
(r_{k-1}-1)+(r_{k-1}-1)+2=2r_{k-1}.
\end{equation*}
\end{itemize}

Therefore, the $k\text{-}th$ iterated total graph of $G$ is a regular graph of degree
$$r_k=2r_{k-1}.$$

\item Let $\mathcal{T}^{k-1}(G)$ be a regular graph with $m_{k-1}$ edges then $m_{k-1}=\frac{n_{k-1}r_{k-1}}{2}$.
Therefore, the order of the $k\text{-}th$ iterated total graph of $G$ is
\begin{equation*}
n_k=m_{k-1}+n_{k-1}=\dfrac{n_{k-1}r_{k-1}}{2}+n_{k-1}=n_{k-1}\Big(\frac{r_{k-1}+2}{2}\Big).
\end{equation*}
\end{enumerate}
\end{proof}

Repeated application of the previous theorem generates the following result.

\begin{corollary}
Let $G$ be a regular graph of order $n_0$ and degree $r_0$, then for $k\geq1$
the $k\text{-}th$ iterated total graph of $G$ is a regular graph of degree $r_k$ and order $n_k$, where
\begin{equation}
r_k=2^kr_0,
\end{equation}
and
\begin{equation}
n_k=\dfrac{n_0}{2^k}\prod\limits_{i=0}^{k-1}\Big(2^ir_0+2\Big).
\end{equation}
\end{corollary}

For the next result, we need the following Lemma seen in \cite{cvetc}.

\begin{lemma} \label{lem1} \cite{cvetc}
Let $G$ be a regular graph of order $n$ and degree $r$. Then the eigenvalues of $\mathcal{T}(G)$ are

\begin{equation}
\left\{\begin{array}{cl}
\dfrac{2\lambda_i+r-2\pm\sqrt{4\lambda_i+r^2+4}}{2} & \text{\ $i=1,2,{\ldots},n,$ and}\\
\\
-2 & \text{$\dfrac{n(r-2)}{2}$ times,}
\end{array}\right.
\end{equation}
where $\lambda_i$ is an eigenvalue of $G$.
\end{lemma}
Now we consider bounds for the incidence energy of iterated total graph.
\begin{theorem}
Let $G$ be a regular graph of order $n$ and degree $r$. Then
\begin{itemize}
\item $IE(\mathcal{T}(G))<\dfrac{n(r-2)\sqrt{2r-2}}{2}+2n\sqrt{r}+(n-1)\sqrt{3r-2},$ and
\item $IE(\mathcal{T}(G))\geq\dfrac{n(r-2)\sqrt{2r-2}}{2}+(n+1)\sqrt{r}+\sqrt{3r-2}+(n-1)\sqrt{2r-2}$.
\end{itemize}
Equality hold if and only if $G\cong K_2$.
\end{theorem}

\begin{proof}
1. Let $r=1$ then $G$ is union disjoint of copies of $K_2$ and $\mathcal{T}(G)$ is union disjoint of copies of $K_3$. This is, $G\cong \frac{n}{2}K_2$ and
$\mathcal{T}(G)\cong \frac{n}{2}K_3$, where $n$ is even. Since $IE(K_3)=4$ it follow that $IE(\mathcal{T}(G))=2n$. Therefore, if $r=1$

$$\dfrac{n(r-2)\sqrt{2r-2}}{2}+2n\sqrt{r}+(n-1)\sqrt{3r-2}=3n-1>2n$$

$$\dfrac{n(r-2)\sqrt{2r-2}}{2}+(n+1)\sqrt{r}+\sqrt{3r-2}+(n-1)\sqrt{2r-2}=n+2\leq 2n,$$
with equality if and only if $n=2,$ this is, $G\cong K_2$.

Let $r\geq2$. Since $G$ is a regular graph of degree $r$. From Theorem \ref{teo1}, $\mathcal{T}(G)$ is a regular graph
of degree $2r$. From Lemma \ref{lem1}, the signless Laplacian eigenvalues of $\mathcal{T}(G)$ are

\begin{equation*}
\left\{\begin{array}{cl}
\dfrac{5r+2\lambda_i-2\pm\sqrt{4\lambda_i+r^2+4}}{2} & \text{\ $i=1,2,{\ldots,}n$ and}\\
\\
2r-2 & \text{$\dfrac{n(r-2)}{2}$ times,}
\end{array}\right.
\end{equation*}
where $\lambda_i$ is an eigenvalue of $G$.

From Perron-Frobenius's Theory, $\lambda_1=r$ and $-r\leq\lambda_i<r$ for $i=2,{\ldots,}n$.
By definition of incidence energy of a graph, we have

\begin{eqnarray*}
IE(\mathcal{T}(G))&=&\dfrac{n(r-2)\sqrt{2r-2}}{2}+2\sqrt{r}+\sqrt{3r-2}\\
&+&\sum_{i=2}^{n}\Bigg(\sqrt{\frac{5r+2\lambda_i-2+\sqrt{4\lambda_i+r^2+4}}{2}}+\sqrt{\frac{5r+2\lambda_i-2-\sqrt{4\lambda_i+r^2+4}}{2}}\Bigg).
\end{eqnarray*}

Let
$$g(t)=\sqrt{\frac{5r+2t-2+\sqrt{r^2+4t+4}}{2}}+\sqrt{\frac{5r+2t-2-\sqrt{r^2+4t+4}}{2}}$$
and
$$h(t)=\sqrt{r^2+4t+4}$$
where $-r\leq t<r$. Thus,

$$g'(t)=\dfrac{1}{\sqrt{2}h(t)}\bigg(\frac{h(t)+1}{\sqrt{5r+2t-2+h(t)}}+\frac{h(t)-1}{\sqrt{5r+2t-2-h(t)}}\bigg).$$

It is clear that $g(t)$ is an increasing function for $-r\leq t<r$.

\begin{eqnarray*}
IE(\mathcal{T}(G))&<&\dfrac{n(r-2)\sqrt{2r-2}}{2}+2\sqrt{r}+\sqrt{3r-2}+\sum_{i=1}^{n}g(r)\\
&=&\dfrac{n(r-2)\sqrt{2r-2}}{2}+2n\sqrt{r}+(n-1)\sqrt{3r-2}.
\end{eqnarray*}

\begin{eqnarray*}
IE(\mathcal{T}(G))&\geq&\dfrac{n(r-2)\sqrt{2r-2}}{2}+2\sqrt{r}+\sqrt{3r-2}+\sum_{i=1}^{n}g(-r)\\
&=&\dfrac{n(r-2)\sqrt{2r-2}}{2}+(n+1)\sqrt{r}+\sqrt{3r-2}+(n-1)\sqrt{2r-2}.
\end{eqnarray*}

Since $n\geq3$ then there exists $\lambda_{i} > -r$ for some $i=2,3,{\ldots,}n$. Therefore, the equality is impossible. 
\end{proof}

\begin{corollary}
Let $G$ be a regular graph of order $n_0$ and degree $r_0\geq2$, and let for $k\geq0$,
the $k\text{-}th$ iterated total graph of $G$ be of degree $r_k$ and  order $n_k$. Then
\begin{itemize}
\item $IE(\mathcal{T}^{k+1}(G))<\dfrac{n_k(r_k-2)\sqrt{2r_k-2}}{2}+2n_k\sqrt{r_k}+(n_k-1)\sqrt{3r_k-2}$, and
\item $IE(\mathcal{T}^{k+1}(G))\geq\dfrac{n_k(r_k-2)\sqrt{2r_k-2}}{2}+(n_k+1)\sqrt{r_k}+\sqrt{3r_k-2}+(n_k-1)\sqrt{2r_k-2}$.
\end{itemize}
Equality hold if and only if $G\cong K_2$ and $k=0$.
\end{corollary}

\begin{corollary}
 Under the notation specified in Corollary $3.1$, for any integer $k$
\begin{itemize}
\item $IE(\mathcal{T}^{k}(G))<(n_k-2n_{k-1})\sqrt{r_k-2}+n_{k-1}\sqrt{2r_k}+(n_{k-1}-1)\sqrt{r_k+r_{k-1}-2}$, and
\item $IE(\mathcal{T}^{k}(G))\geq(n_k-2n_{k-1})\sqrt{r_k-2}+(n_{k-1}+1)\sqrt{r_{k-1}}+\sqrt{r_k+r_{k-1}-2}+(n_{k-1}-1)\sqrt{r_k-2}$.
\end{itemize}
Equality hold if and only if $G\cong K_2$ and $k=1$.
\end{corollary}

\subsection{An application: Constructing nonisomorphic cospectral graphs}

Many constructions of cospectral graphs are known. Most constructions from before
1988 can be found in \cite{DM2}, \S 6.1, and \cite{DM1}, \S 1.3; see also \cite{DM3}, \S 4.6.
More recent constructions of cospectral graphs are presented by Seress \cite{DM4}, who gives an
infinite family of cospectral 8-regular graphs. Graphs cospectral to distance-regular
graphs can be found in \cite{DM5,DM2,DM1,DM3,DM7,DM4,DM6}. Notice that the graphs
mentioned are regular, so they are cospectral with respect to any generalized adjacency
matrix, which in this case includes the Laplace matrix.

Let's consider the functions
\begin{equation*}
f_1(x)=\dfrac{1}{2}(2x+r-2+\sqrt{4x+r^2+4})
\end{equation*}
and
\begin{equation*}
f_2(x)=\dfrac{1}{2}(2x+r-2-\sqrt{4x+r^2+4}).
\end{equation*}
\begin{theorem}\label{teo9}
  Let $G_1$ and $G_2$ be two regular graphs of the same order and degree $n_0$ and $r_0\geq3$, respectively. Then, for any
  $k\geq1$ the following hold
  \begin{itemize}
  \item[(a)] $\mathcal{T}^k(G_1)$ and $\mathcal{T}^k(G_2)$ are of the same order, and have the same number of edges.
  \item[(b)] $\mathcal{T}^k(G_1)$ and $\mathcal{T}^k(G_2)$ are cospectral if and only if $G_1$ and $G_2$ are cospectral.
  \end{itemize}
\end{theorem}

\begin{proof}
Statement (a) follows from Eqs. (2) and (3), and the fact that the number of vertices and edges of $\mathcal{T}^k(G)$ {corresponds to}
the number of vertices of $\mathcal{T}^{k+1}(G)$. Statement (b) follows from relation (4) and the injectivity of the
functions $f_1$ and $f_2$ on the segment $[-r,r],$ \ $r>2$, applied a sufficient number of times. 
\end{proof}

%
%


\vspace{1cm}
{\bf Acknowledgments:}
Eber Lenes was supported by the Research Department of the Universidad del Sin\'u, Colombia, Exequiel Mallea-Zepeda was supported
by Proyecto UTA-Mayor 4740-18, Universidad de Tarapac\'a, Chile and Maria Robbiano was partially supported by project VRIDT UCN 170403003.

\end{document}